\documentclass{amsproc}
\usepackage{amssymb}
\usepackage{amsmath}
\usepackage{amsfonts}

\theoremstyle{plain}

\newtheorem{corollary}{Corollary}

\newtheorem{definition}{Definition}

\newtheorem{notation}{Notation}

\newtheorem{proposition}{Proposition}
\newtheorem{remark}{Remark}

\newtheorem{theorem}{Theorem}
\numberwithin{equation}{section}

\input{tcilatex}

\begin{document}
\title[]{Some Remarks on Nonlinear Hyperbolic Equations}
\author{Kamal N. Soltanov}
\address{{\small Department of Mathematics, }\\
{\small Faculty of Sciences, Hacettepe University, }\\
{\small Beytepe, Ankara, TR-06532, TURKEY} }
\email{soltanov@hacettepe.edu.tr}
\date{}

\begin{abstract}
Here a mixed problem for a nonlinear hyperbolic equation with Neumann
boundary value condition is investigated, and a priori estimations for the
possible solutions of the considered problem are obtained. These results
demonstrate that any solution of this problem possess certain smoothness
properties.
\end{abstract}

\maketitle

\subsubsection{Introduction}

In this article we consider a mixed problem for a nonlinear hyperbolic
equation and study the smoothness of a possible solution of the problem, in
some sense. Here we got some new a priori estimations for a solution of the
considered problem.

It is known that, up to now, the problem of the solvability of a nonlinear
hyperbolic equation with nonlinearity of this type has not been solved when $%
\Omega \subset R^{n}$, $n\geq 2$. It should also be noted that it is not
possible to use the a priori estimations, which can be obtained by the known
methods, to prove the solvability in this case. Consequently, there are no
obtained results on a solvability of a mixed problem for the equation of the
following type

\begin{equation*}
\frac{\partial ^{2}u}{\partial t^{2}}-\overset{n}{\underset{i=1}{\sum }}D_{i}%
\left[ a_{i}\left( t,x\right) \left\vert D_{i}u\right\vert ^{p-2}D_{i}u%
\right] =h\left( t,x\right) ,\quad \left( t,x\right) \in Q_{T}
\end{equation*}%
\begin{equation*}
Q_{T}\equiv \left( 0,T\right) \times \Omega ,\quad \Omega \subset
R^{n},\quad T>0,\quad p>2.
\end{equation*}%
As known, the investigation of a mixed problem for the nonlinear hyperbolic
equations of such type on the Sobolev type spaces when $\ \Omega \subset
R^{n}$, $n\geq 2$ is connected with many difficulties (see, for example, the
works of Leray, Courant, Friedrichs, Lax, F. John, Garding, Ladyzhenskaya,
J.-L. Lions, H. Levine, Rozdestvenskii and also, [2, 7 - 11, 14 - 16, 18,
19], etc. ). Furthermore the possible solutions of this problem may possess
a gradient catastrophe. Only in the case $n=1$, it is achieved to prove
solvability theorems for the problems of such type (and essentially with
using the Riemann invariants).

However, recently certain classes of nonlinear hyperbolic equations were
investigated and results on the solvability of the considered problems in a
more generalized sense were obtained (see, for example, [13] its references)
and also certain result about dense solvability was obtained ([20]).
Furthermore there are such special class of the nonlinear hyperbolic
equations, for which the solvabilities were studied under some additional
conditions (see, [3, 4, 13, 17, 23] and its references), for example, under
some geometrical conditions.

Here, we investigate a mixed problem for equations of certain class with the
Neumann boundary-value conditions. In the beginning, a mixed problem for a
nonlinear parabolic equation with similar nonlinearity and conditions as
above is studied and the existence of the strongly solutions of this problem
is proved. Further some a priori estimations for a possible solution of the
considered problem is received in the hyperbolic case with use of the result
on the parabolic problem studied above. These results demonstrate that any
solution of our main problem possesses certain smoothness properties, which
might help for the proof of some existence theorems.\footnote{%
Unfortunately, I could not use the obtained estimates for this aim.}

\section{Formulation of Problem}

Consider the problem 
\begin{equation}
\frac{\partial ^{2}u}{\partial t^{2}}-\overset{n}{\underset{i=1}{\sum }}%
D_{i}\left( \left\vert D_{i}u\right\vert ^{p-2}D_{i}u\right) =h\left(
t,x\right) ,\quad \left( t,x\right) \in Q_{T},\quad p>2,  \tag{1.1}
\end{equation}%
\begin{equation}
u\left( 0,x\right) =u_{0}\left( x\right) ,\quad \frac{\partial u}{\partial t}%
\ \left\vert ~_{t=0}\right. =u_{1}\left( x\right) ,\quad x\in \Omega \subset
R^{n},\quad n\geq 2,  \tag{1.2}
\end{equation}%
\begin{equation}
\frac{\partial u}{\partial \widehat{\nu }}\left\vert {}~_{\Gamma }\right.
\equiv \overset{n}{\underset{i=1}{\sum }}\left\vert D_{i}u\right\vert
^{p-2}D_{i}u\cos \left( \nu ,x_{i}\right) =0,\quad \left( x%
{\acute{}}%
,t\right) \in \Gamma \equiv \partial \Omega \times \left[ 0,T\right] , 
\tag{1.3}
\end{equation}%
here $\Omega \subset R^{n},n\geq 2$ be a bounded domain with sufficiently
smooth boundary $\partial \Omega $; $u_{0}\left( x\right) $, $u_{1}\left(
x\right) $, $h\left( t,x\right) $ are functions such that $u_{0},u_{1}\in
W_{p}^{1}\left( \Omega \right) $, $h\in L_{p}\left( 0,T;W_{p}^{1}\left(
\Omega \right) \right) $, $\nu $ denote the unit outward normal to $\partial
\Omega $ (see, [10, 12]).

Introduce the class of the functions $u:Q\longrightarrow R$ 
\begin{equation*}
V\left( Q\right) \equiv W_{2}^{1}\left( 0,T;L_{2}\left( \Omega \right)
\right) \cap L^{\infty }\left( 0,T;W_{p}^{1}\left( \Omega \right) \right)
\cap L_{p-1}\left( 0,T;\widetilde{S}_{1,2\left( p-2\right) ,2}^{1}\left(
\Omega \right) \right) \cap
\end{equation*}%
\begin{equation*}
\left\{ u\left( t,x\right) \left\vert \ \frac{\partial ^{2}u}{\partial t^{2}}%
,\ \underset{i=1}{\overset{n}{\sum }}\left( \left\vert D_{i}u\right\vert
^{p-2}D_{i}^{2}u\right) \in L_{1}\left( 0,T;L_{2}\left( \Omega \right)
\right) \right. \right\} \cap
\end{equation*}%
\begin{equation*}
\left\{ u\left( t,x\right) \left\vert \ \overset{n}{\underset{i=1}{\sum }}%
\underset{0}{\overset{t}{\int }}\left\vert D_{i}u\right\vert
^{p-2}D_{i}ud\tau \in W_{\infty }^{1}\left( 0,T;L_{q}\left( \Omega \right)
\right) \cap L^{\infty }\left( 0,T;W_{2}^{1}\left( \Omega \right) \right)
\right. \right\}
\end{equation*}%
\begin{equation}
\left\{ u\left( t,x\right) \left\vert \ u\left( 0,x\right) =u_{0}\left(
x\right) ,\ \frac{\partial u}{\partial t}\ \left\vert ~_{t=0}\right.
=u_{1}\left( x\right) ,\ \frac{\partial u}{\partial \widehat{\nu }}%
\left\vert ~_{\Gamma }\right. =0\right. \right\}  \tag{DS}
\end{equation}

where 
\begin{equation*}
\widetilde{S}_{1,\alpha ,\beta }^{1}\left( \Omega \right) \equiv \left\{
u\left( t,x\right) \left\vert ~\left[ u\right] _{S_{1,\alpha ,\beta
}^{1}}^{\alpha +\beta }=\left\Vert u\right\Vert _{\alpha +\beta }^{\alpha
+\beta }+\underset{i=1}{\overset{n}{\sum }}\left\Vert D_{i}u\right\Vert
_{\alpha +\beta }^{\alpha +\beta }+\right. \right.
\end{equation*}

\begin{equation*}
\left. \left\Vert \underset{i,j=1}{\overset{n}{\sum }}\left\vert
D_{i}u\right\vert ^{\frac{\alpha }{\beta }}D_{j}D_{i}u\right\Vert _{\beta
}^{\beta }<\infty \right\} ,\quad \alpha \geq 0,\ \beta \geq 1.
\end{equation*}

Thus, we will understand the solution of the problem in the following form:
A function $u\left( t,x\right) \in V\left( Q\right) $ is called solution of
problem (1.1) - (1.3) if $u\left( t,x\right) $ satisfies the following
equality 
\begin{equation*}
\left[ \frac{\partial ^{2}u}{\partial t^{2}},v\right] -\left[ \overset{n}{%
\underset{i=1}{\sum }}D_{i}\left( \left\vert D_{i}u\right\vert
^{p-2}D_{i}u\right) ,v\right] =\left[ h,v\right]
\end{equation*}%
for any $v\in W_{q}^{1}\left( 0,T;L_{2}\left( \Omega \right) \right) \cap
L^{\infty }\left( Q\right) $, where $\left[ \circ ,\circ \right] \equiv 
\underset{Q}{\int }\circ \times \circ ~dxdt$.

Our aim in this article is to prove

\begin{theorem}
Under the conditions of this section each solution of problem (1.1)-(1.3)
belongs to a bounded subset of the space $V\left( Q\right) $ defined in (DS).
\end{theorem}

For the investigation of the posed problem in the beginning we will study
two problems, which are connected with considered problem. One of \ these
problems immediately follows from problem (1.1)-(1.3) and have the form: 
\begin{equation}
\frac{\partial u}{\partial t}-\overset{n}{\underset{i=1}{\sum }}D_{i}\overset%
{t}{\underset{0}{\int }}\left( \left\vert D_{i}u\right\vert
^{p-2}D_{i}u\right) d\tau =H\left( t,x\right) +u_{1}\left( x\right) , 
\tag{1.4}
\end{equation}%
where $H\left( t,x\right) =\underset{0}{\overset{t}{\int }}h\left( \tau
,x\right) d\tau $.

Consequently, if $u\left( t,x\right) $ is a solution of problem (1.1) -
(1.3) then $u\left( t,x\right) $ is a such solution of equation (1.4) that
the following conditions are fulfilled: 
\begin{equation}
u\left( 0,x\right) =u_{0}\left( x\right) ,\quad \frac{\partial u}{\partial 
\widehat{\nu }}\left\vert _{\Gamma }\right. =0.  \tag{1.5}
\end{equation}

From here it follows that problems (1.1) - (1.3) and (1.4) - (1.5) are
equivalent.

And other problem is the nonlinear parabolic problem 
\begin{equation}
\frac{\partial u}{\partial t}-\ \overset{n}{\underset{i=1}{\sum }}%
D_{i}\left( \left\vert D_{i}u\right\vert ^{p-2}D_{i}u\right) =h\left(
t,x\right) ,\quad \left( t,x\right) \in Q,\ p>2,\ n\geq 2,  \tag{1.6}
\end{equation}%
\begin{equation}
u\left( 0,x\right) =u_{0}\left( x\right) ,\quad x\in \Omega ,\ \quad \frac{%
\partial u}{\partial \widehat{\nu }}\left\vert ~_{\Gamma }\right. =0, 
\tag{1.7}
\end{equation}%
where $u_{0}\in W_{p}^{1}\left( \Omega \right) $, $h\in L_{2}\left(
0,T;W_{2}^{1}\left( \Omega \right) \right) $ and $p>2$.

In the beginning the solvability of this problem is studied, for which the
general result is used, therefore we begin with this result.

\section{Some General Solvability Results}

Let $X,Y$ be locally convex vector topological spaces, $B\subseteq Y$ be a
Banach space and $g:D\left( g\right) \subseteq X\longrightarrow Y$ be a
mapping. Introduce the following subset of $\ X$ 
\begin{equation*}
\mathcal{M}_{gB}\equiv \left\{ x\in X\left\vert ~g\left( x\right) \in
B,\right. \func{Im}g\cap B\neq \varnothing \right\} .
\end{equation*}

\begin{definition}
A subset $\mathcal{M}\subseteq X$ is called a $pn-$space (i.e. pseudonormed
space) if $\mathcal{M}$ is a topological space and there is a function $%
\left[ \cdot \right] _{S}:\mathcal{M}\longrightarrow R_{+}^{1}\equiv \left[
0,\infty \right) $ (wh\i ch is called $p-$norm of $\mathcal{M}$) such that

qn) $\left[ x\right] _{\mathcal{M}}\geq 0$, $\forall x\in \mathcal{M}$ and $%
0\in \mathcal{M}$, $x=0\Longrightarrow \left[ x\right] _{\mathcal{M}}=0$;

pn) \ $\left[ x_{1}\right] _{\mathcal{M}}\neq \left[ x_{2}\right] _{\mathcal{%
M}}\Longrightarrow x_{1}\neq x_{2}$, for $x_{1},x_{2}\in \mathcal{M}$, and $%
\left[ x\right] _{\mathcal{M}}=0\Longrightarrow x=0$;
\end{definition}

The following conditions are often fulfilled in the spaces $\mathcal{M}_{gB}$%
.

N) There exist a convex function $\nu :R^{1}\longrightarrow \overline{%
R_{+}^{1}}$ and number $K\in \left( 0,\infty \right] $ such that $\left[
\lambda x\right] _{\mathcal{M}}\leq \nu \left( \lambda \right) \left[ x%
\right] _{\mathcal{M}}$ for any $x\in \mathcal{M}$ and $\lambda \in R^{1}$, $%
\left\vert \lambda \right\vert <K$, moreover $\underset{\left\vert \lambda
\right\vert \longrightarrow \lambda _{j}}{\lim }\frac{\nu \left( \lambda
\right) }{\left\vert \lambda \right\vert }=c_{j}$, $j=0,1$ where $\lambda
_{0}=0$, $\lambda _{1}=K$ and $c_{0}=c_{1}=1$ or $c_{0}=0$, $c_{1}=\infty $,
i.e. if $K=\infty $ then $\lambda x\in \mathcal{M}$ for any $x\in S$ and $%
\lambda \in R^{1}.$

Let $g:D\left( g\right) \subseteq X\longrightarrow Y$ be such a mapping that 
$\mathcal{M}_{gB}\neq \varnothing $ and the following conditions are
fulfilled

(g$_{\text{1}}$) $g:D\left( g\right) \longleftrightarrow \func{Im}g$ is
bijection and $g\left( 0\right) =0$;

(g$_{\text{2}}$) there is a function $\nu :R^{1}\longrightarrow \overline{%
R_{+}^{1}}$ satisfying condition N such that 
\begin{equation*}
\left\Vert g\left( \lambda x\right) \right\Vert _{B}\leq \nu \left( \lambda
\right) \left\Vert g\left( x\right) \right\Vert _{B},\ \forall x\in \mathcal{%
M}_{gB},\ \forall \lambda \in R^{1};
\end{equation*}%
If mapping $g$ satisfies conditions (g$_{1}$) and (g$_{2}$), then $\mathcal{M%
}_{gB}$ is a $pn-$space with $p-$norm defined in the following way: there is
a one-to-one function $\psi :R_{+}^{1}\longrightarrow R_{+}^{1}$, $\psi
\left( 0\right) =0$, $\psi ,\psi ^{-1}\in C^{0}$ such that $\left[ x\right]
_{\mathcal{M}_{gB}}\equiv \psi ^{-1}\left( \left\Vert g\left( x\right)
\right\Vert _{B}\right) $. In this case $\mathcal{M}_{gB}$ is a metric space
with a metric: $d_{\mathcal{M}}\left( x_{1};x_{2}\right) \equiv \left\Vert
g\left( x_{1}\right) -g\left( x_{2}\right) \right\Vert _{B}$. Further, we
consider just such type of $pn-$spaces.

\begin{definition}
The $pn-$space $\mathcal{M}_{gB}$ is called weakly complete if $g\left( 
\mathcal{M}_{gB}\right) $ is weakly closed in $B.$ The pn-space $\mathcal{M}%
_{gB}$ is "reflexive" if each bounded weakly closed subset of $\mathcal{M}%
_{gB}$ is weakly compact in $\mathcal{M}_{gB}$.
\end{definition}

It is clear that if $B$ is a reflexive Banach space and $\mathcal{M}_{gB}$
is a $pn-$space, then $\mathcal{M}_{gB}$ is "reflexive". Moreover, if $B$ is
a separable Banach space, then $\mathcal{M}_{gB}$ is separable (see, for
example, [21, 22] and their references).

Now, consider a nonlinear equation in the general form. Let $X,Y$ be Banach
spaces with dual spaces $X^{\ast },Y^{\ast }$ respectively, $\mathcal{M}%
_{0}\subseteq X$ is a weakly complete $pn-$space, $f:D\left( f\right)
\subseteq X\longrightarrow Y$ be a nonlinear operator. Consider the equation 
\begin{equation}
f\left( x\right) =y,\quad y\in Y.  \tag{2.1}
\end{equation}

\begin{notation}
It is clear that (2.1)$\mathit{\ }$is equivalent to the following functional
equation:%
\begin{equation}
\left\langle f\left( x\right) ,y^{\ast }\right\rangle =\left\langle
y,y^{\ast }\right\rangle ,\quad \forall y^{\ast }\in Y^{\ast }.  \tag{2.2}
\end{equation}
\end{notation}

Let $f:D\left( f\right) \subseteq X\longrightarrow Y$ be a nonlinear bounded
operator and the following conditions hold

1) $f:\mathcal{M}_{0}\subseteq D\left( f\right) \longrightarrow Y$ is a
weakly compact (weakly "continuous") mapping, i.e. for any weakly
convergence sequence $\left\{ x_{m}\right\} _{m=1}^{\infty }\subset \mathcal{%
M}_{0}$ in $\mathcal{M}_{0}$ (i.e. $x_{m}\overset{\mathcal{M}_{0}}{%
\rightharpoonup }x_{0}\in \mathcal{M}_{0}$) there is a subsequence $\left\{
x_{m_{k}}\right\} _{k=1}^{\infty }\subseteq \left\{ x_{m}\right\}
_{m=1}^{\infty }$ such that $f\left( x_{m_{k}}\right) \overset{Y}{%
\rightharpoonup }f\left( x_{0}\right) $ weakly in $Y$ (or for a general
sequence if $\mathcal{M}_{0}$ is not a separable space) and $\mathcal{M}_{0}$
is a weakly complete $pn-$space;

2) there exist a mapping $g:X_{0}\subseteq X\longrightarrow Y^{\ast }$ and a
continuous function $\varphi :R_{+}^{1}\longrightarrow R^{1}$ nondecreasing
for $\tau \geq \tau _{0}\geq 0$ \& $\varphi \left( \tau _{1}\right) >0$ for
a number $\tau _{1}>0,$ such that $g$ generates a "coercive" pair with $f$ \
in a generalized sense on the topological space $X_{1}\subseteq X_{0}\cap 
\mathcal{M}_{0}$, i.e. 
\begin{equation*}
\left\langle f\left( x\right) ,g\left( x\right) \right\rangle \geq \varphi
\left( \lbrack x]_{\mathcal{M}_{0}}\right) [x]_{\mathcal{M}_{0}},\quad
\forall x\in X_{1},
\end{equation*}%
where $X_{1}$ is a topological space such that $\overline{X_{1}}%
^{X_{0}}\equiv X_{0}$ and $\overline{X_{1}}^{\mathcal{M}_{0}}\equiv \mathcal{%
M}_{0}$, and\textit{\ }$\left\langle \cdot ,\cdot \right\rangle $ is a%
\textit{\ }dual form of the pair $\left( Y,Y^{\ast }\right) $. Moreover one
of the following conditions $\left( \alpha \right) $ or $\left( \beta
\right) $ holds:

$\left( \alpha \right) $ if $g\equiv L$ is a linear continuous operator,
then $\mathcal{M}_{0}$ is a \textquotedblright reflexive\textquotedblright\
space (see [21, 22]), $X_{0}\equiv X_{1}\subseteq \mathcal{M}_{0}$ is a
separable topological vector space which is dense in $\mathcal{M}_{0}$ and $%
\ker L^{\ast }=\left\{ 0\right\} $.

$\left( \beta \right) $ if $g$ is a bounded operator (in general,
nonlinear), then $Y$ is a reflexive separable space, $g\left( X_{1}\right) $
contains an everywhere dense linear manifold of $Y^{\ast }$ and $g^{-1}$ is
a weakly compact (weakly continuous) operator from $Y^{\ast }$ to $\mathcal{M%
}_{0}$.

\begin{theorem}
\textit{Let conditions 1 and 2 hold. Then equation (2.1) (or (2.2))\ is
solvable in }$\mathcal{M}_{0}$\textit{\ for any} $y\in Y$ \textit{satisfying
the following inequation: there exists} $r>0$ such that\textit{\ the
inequality} 
\begin{equation}
\varphi \left( \lbrack x]_{\mathcal{M}_{0}}\right) [x]_{\mathcal{M}_{0}}\geq
\left\langle y,g\left( x\right) \right\rangle ,\text{ for}\quad \forall x\in
X_{1}\quad \text{with}\quad \lbrack x]_{\mathcal{M}}\geq r.  \tag{2.3}
\end{equation}%
holds.
\end{theorem}

\begin{proof}
Assume that conditions 1 and 2 ($\alpha $)\textit{\ }are fulfilled and $y\in
Y$ such that (2.3) holds. We are going to use Galerkin's approximation
method. Let $\left\{ x^{k}\right\} _{k=1}^{\infty }$ be a complete system in
the (separable) space $X_{1}\equiv X_{0}$. Then we are looking for
approximate solutions in the form $x_{m}=\overset{m}{\underset{k=1}{\sum }}%
c_{mk}x^{k},$ where $c_{mk}$ are unknown coefficients, which can be
determined from the system of algebraic equations 
\begin{equation}
\Phi _{k}\left( c_{m}\right) :=\left\langle f\left( x_{m}\right) ,g\left(
x^{k}\right) \right\rangle -\left\langle y,g\left( x^{k}\right)
\right\rangle =0,\quad k=1,2,...,m  \tag{2.4}
\end{equation}%
where $c_{m}\equiv \left( c_{m1},c_{m2},...,c_{mm}\right) $.

We observe that the mapping \ $\Phi \left( c_{m}\right) :=\left( \Phi
_{1}\left( c_{m}\right) ,\Phi _{2}\left( c_{m}\right) ,...,\Phi _{m}\left(
c_{m}\right) \right) $ is continuous by virtue of condition 1. (2.3) implies
the existence of such $r=r\left( \left\Vert y\right\Vert _{Y}\right) >0$
that the \textquotedblright acute angle\textquotedblright\ condition is
fulfilled for all $x_{m}$ with $\left[ x_{m}\right] _{\mathcal{M}_{0}}\geq r$
, i.e. for any $c_{m}\in S_{r_{1}}^{R^{m}}\left( 0\right) \subset R^{m}$, $%
r_{1}\geq r$ the inequality 
\begin{equation*}
\overset{m}{\underset{k=1}{\sum }}\left\langle \Phi _{k}\left( c_{m}\right)
,c_{mk}\right\rangle \equiv \left\langle f\left( x_{m}\right) ,g\left( 
\overset{m}{\underset{k=1}{\sum }}c_{mk}x^{k}\right) \right\rangle
-\left\langle y,g\left( \overset{m}{\underset{k=1}{\sum }}c_{mk}x^{k}\right)
\right\rangle =\quad
\end{equation*}%
\begin{equation*}
\left\langle f\left( x_{m}\right) ,g\left( x_{m}\right) \right\rangle
-\left\langle y,g\left( x_{m}\right) \right\rangle \geq 0,\quad \forall
c_{m}\in 
\mathbb{R}
^{m},\left\Vert c_{m}\right\Vert _{%
\mathbb{R}
^{m}}=r_{1}.
\end{equation*}%
holds. The solvability of system (2.4) for each $m=1,2,\ldots $ follows from
a well-known \textquotedblleft acute angle\textquotedblright\ lemma ([10, 21
- 23]),\ which is equivalent to the Brouwer's fixed-point theorem. Thus, $%
\left\{ x_{m}\left\vert ~m\geq \right. 1\right\} $ is the sequence of
approximate solutions which are contained in a bounded subset of the space $%
\mathcal{M}_{0}$. Further arguments are analogous to those from [10, 23]
therefore we omit them. It remains to pass to the limit in (2.4)\ by $m$ and
use the weak convergency of a subsequence of the sequence $\left\{
x_{m}\left\vert ~m\geq \right. 1\right\} $, the weak compactness of the
mapping $f$, and the completeness of the system $\left\{ x^{k}\right\}
_{k=1}^{\infty }$in the space $X_{1}$.

Hence we get the limit element $x_{0}=w-\underset{j\nearrow \infty }{\lim }%
x_{m_{j}}\in S_{0}$ which is the solution of the equation 
\begin{equation}
\left\langle f\left( x_{0}\right) ,g\left( x\right) \right\rangle
=\left\langle y,g\left( x\right) \right\rangle ,\quad \forall x\in X_{0}, 
\tag{2.5}
\end{equation}%
or of the equation 
\begin{equation}
\left\langle g^{\ast }\circ f\left( x_{0}\right) ,x\right\rangle
=\left\langle g^{\ast }\circ y,x\right\rangle ,\quad \forall x\in X_{0}. 
\tag{2.5'}
\end{equation}

In the second case, i.e. when conditions 1 and 2 ($\beta $)\ are fulfilled
and $y\in Y$ such that (2.3) holds, we suppose that the approximate
solutions are searched in the form 
\begin{equation}
x_{m}=g^{-1}\left( \overset{m}{\underset{k=1}{\sum }}c_{mk}y_{k}^{\ast
}\right) \equiv g^{-1}\left( y_{\left( m\right) }^{\ast }\right) ,\quad
i.e.\ x_{m}=g^{-1}\left( y_{\left( m\right) }^{\ast }\right)  \tag{2.6}
\end{equation}%
where $\left\{ y_{k}^{\ast }\right\} _{k=1}^{\infty }\subset Y^{\ast }$ is a
complete system in the (separable) space $Y^{\ast }$ and belongs to $g\left(
X_{1}\right) $. In this case the unknown coefficients $c_{mk}$, that might
be determined from the system of algebraic equations 
\begin{equation}
\widetilde{\Phi }_{k}\left( c_{m}\right) :=\left\langle f\left( x_{m}\right)
,y_{k}^{\ast }\right\rangle -\left\langle y,y_{k}^{\ast }\right\rangle
=0,\quad k=1,2,...,m  \tag{2.7}
\end{equation}%
with $c_{m}\equiv \left( c_{m1},c_{m2},...,c_{mm}\right) $, from which under
the our conditions we get 
\begin{equation}
\left\langle f\left( x_{m}\right) ,y_{k}^{\ast }\right\rangle -\left\langle
y,y_{k}^{\ast }\right\rangle =\left\langle f\left( g^{-1}\left( y_{\left(
m\right) }^{\ast }\right) \right) ,y_{k}^{\ast }\right\rangle -\left\langle
y,y_{k}^{\ast }\right\rangle =0,  \tag{2.7'}
\end{equation}%
for $k=1,2,...,m$.

As above we observe that the mapping \ 
\begin{equation*}
\widetilde{\Phi }\left( c_{m}\right) :=\left( \widetilde{\Phi }_{1}\left(
c_{m}\right) ,\widetilde{\Phi }_{2}\left( c_{m}\right) ,...,\widetilde{\Phi }%
_{m}\left( c_{m}\right) \right)
\end{equation*}%
is continuous by virtue of conditions 1 and 2($\beta $). (2.3) implies the
existence of such $\widetilde{r}>0$ that the \textquotedblright acute
angle\textquotedblright\ condition is fulfilled for all $y_{\left( m\right)
}^{\ast }$ with $\left\Vert y_{\left( m\right) }^{\ast }\right\Vert
_{Y^{\ast }}\geq \widetilde{r}$ , i.e. for any $c_{m}\in
S_{r_{1}}^{R^{m}}\left( 0\right) \subset R^{m}$, $\widetilde{r}_{1}\geq 
\widetilde{r}$ the inequality 
\begin{equation*}
\overset{m}{\underset{k=1}{\sum }}\left\langle \widetilde{\Phi }_{k}\left(
c_{m}\right) ,c_{mk}\right\rangle \equiv \left\langle f\left( x_{m}\right) ,%
\overset{m}{\underset{k=1}{\sum }}c_{mk}y_{k}^{\ast }\right\rangle
-\left\langle y,\overset{m}{\underset{k=1}{\sum }}c_{mk}y_{k}^{\ast
}\right\rangle =\quad
\end{equation*}%
\begin{equation*}
\left\langle f\left( g^{-1}\left( y_{\left( m\right) }^{\ast }\right)
\right) ,y_{\left( m\right) }^{\ast }\right\rangle -\left\langle y,y_{\left(
m\right) }^{\ast }\right\rangle =\left\langle f\left( x_{m}\right) ,g\left(
x_{m}\right) \right\rangle -\left\langle y,g\left( x_{m}\right)
\right\rangle \geq 0,
\end{equation*}%
\begin{equation*}
\forall c_{m}\in 
\mathbb{R}
^{m},\left\Vert c_{m}\right\Vert _{%
\mathbb{R}
^{m}}=\widetilde{r}_{1}.
\end{equation*}%
holds by virtue of the conditions. Consequently the solvability of the
system (2.7) (or (2.7')) for each $m=1,2,\ldots $ follows from the
\textquotedblleft acute angle\textquotedblright\ lemma, as above. Thus, we
obtain $\left\{ y_{\left( m\right) }^{\ast }\left\vert ~m\geq \right.
1\right\} $ is the sequence of the approximate solutions of system (2.7'),
that is contained in a bounded subset of $Y^{\ast }$. From here it follows
there is a subsequence $\left\{ y_{\left( m_{j}\right) }^{\ast }\right\}
_{j=1}^{\infty }$ of the sequence $\left\{ y_{\left( m\right) }^{\ast
}\left\vert ~m\geq \right. 1\right\} $ such that it is weakly convergent in $%
Y^{\ast }$, and consequently the sequence $\left\{ x_{m_{j}}\right\}
_{j=1}^{\infty }\equiv \left\{ g^{-1}\left( y_{\left( m_{j}\right) }^{\ast
}\right) \right\} _{j=1}^{\infty }$ weakly converges in the space $\mathcal{M%
}_{0}$ by the condition 2($\beta $) (maybe after passing to the subsequence
of it). It remains to pass to the limit in (2.7')\ by $j$ and use a weak
convergency of a subsequence of the sequence $\left\{ y_{\left( m\right)
}^{\ast }\left\vert ~m\geq \right. 1\right\} $, the weak compactness of
mappings $f$ and $g^{-1}$, and next the completeness of the system $\left\{
y_{k}^{\ast }\right\} _{k=1}^{\infty }$in the space $Y^{\ast }$.

Hence we get the limit element $x_{0}=w-\underset{j\nearrow \infty }{\lim }%
x_{m_{j}}$ $=w-\underset{j\nearrow \infty }{\lim }g^{-1}\left( y_{\left(
m_{j}\right) }^{\ast }\right) \in \mathcal{M}_{0}$ and it is the solution of
the equation 
\begin{equation}
\left\langle f\left( x_{0}\right) ,y^{\ast }\right\rangle =\left\langle
y,y^{\ast }\right\rangle ,\quad \forall y^{\ast }\in Y^{\ast }.  \tag{2.8}
\end{equation}%
Q.E.D.\footnote{%
See also, Soltanov K.N., On Noncoercive Semilinear Equations, Journal- NA:
Hybrid Systems, (2008), 2, 2, 344-358.}
\end{proof}

\begin{remark}
\textit{It is obvious that if there exists a function }$\psi
:R_{+}^{1}\longrightarrow R_{+}^{1}$, $\psi \in C^{0}$\textit{such that }$%
\psi \left( \xi \right) =0\Longleftrightarrow \xi =0$\textit{\ and if the
following inequality\ is fulfilled }$\left\Vert x_{1}-x_{2}\right\Vert
_{X}\leq \psi \left( \left\Vert f\left( x_{1}\right) -f\left( x_{2}\right)
\right\Vert _{Y}\right) $\textit{\ for all }$x_{1},x_{2}\in \mathcal{M}_{0}$%
\textit{\ then solution of equation (2.2) is unique. }
\end{remark}

\begin{notation}
It should be noted the spaces of the $pn-$space type often arising from
nonlinear problems with nonlinear main parts, for example,

1) the equation of the nonlinear filtration or diffusion that have the
expression: 
\begin{equation*}
\frac{\partial u}{\partial t}-\nabla \cdot \left( g\left( u\right) \nabla
u\right) +h\left( t,x,u\right) =0,\quad u\left\vert \ _{\partial \Omega
\times \left[ 0.T\right] }\right. =0,
\end{equation*}%
\begin{equation*}
u\left( 0,x\right) =u_{0}\left( x\right) ,\quad x\in \Omega \subset 
\mathbb{R}
^{n},\quad n\geq 1
\end{equation*}%
where $g:%
\mathbb{R}
\longrightarrow 
\mathbb{R}
_{+}$ is a convex function ($g\left( s\right) \equiv \left\vert s\right\vert
^{\rho }$, $\rho >0$) and $h\left( t,x,s\right) $ is a Caratheodory
function, in this case it is needed to investigate 
\begin{equation*}
S_{1,\rho ,2}\left( \Omega \right) \equiv \left\{ u\in L^{1}\left( \Omega
\right) \left\vert \ \underset{\Omega }{\dint }g\left( u\left( x\right)
\right) \left\vert \nabla u\right\vert ^{2}dx\right. <\infty ;\ \
u\left\vert \ _{\partial \Omega }\right. =0\right\} ;
\end{equation*}

2) the equation of the Prandtl-von Mises type equation that have the
expression: 
\begin{equation*}
\frac{\partial u}{\partial t}-\left\vert u\right\vert ^{\rho }\Delta
u+h\left( t,x,u\right) =0,\quad u\left\vert \ _{\partial \Omega \times \left[
0.T\right] }\right. =0,
\end{equation*}%
\begin{equation*}
u\left( 0,x\right) =u_{0}\left( x\right) ,\quad x\in \Omega \subset 
\mathbb{R}
^{n},\quad n\geq 1
\end{equation*}%
where $\rho >0$ and $h\left( t,x,s\right) $ is a Caratheodory function, in
this case it is needed to investigate the spaces of the following spaces
type $S_{1,\mu ,q}\left( \Omega \right) $ ($\mu \geq 0,q\geq 1$) and 
\begin{equation*}
S_{\Delta ,\rho ,2}\left( \Omega \right) \equiv \left\{ u\in L^{1}\left(
\Omega \right) \left\vert \ \underset{\Omega }{\dint }\left\vert u\left(
x\right) \right\vert ^{\rho }\left\vert \Delta u\right\vert ^{2}dx\right.
<\infty ;\ \ u\left\vert \ _{\partial \Omega }\right. =0\right\}
\end{equation*}%
etc. \footnote{%
Theorem and the spaces of such type were used earlier in many works see, for
example, [10, 21], and also the following articles with therein references:
\par
Ju. A. Dubinskii - Weakly convergence into nonlinear elliptic and parabolic
equations, Matem. Sborn., (1965), 67, n. 4; Soltanov K.N. - On solvability
some nonlinear parabolic problems with nonlinearitygrowing quickly
po-lynomial functions. Matematczeskie zametki, 32, 6, 1982.
\par
Soltanov K.N. : Some embedding theorems and its applications to nonlinear
equations. Differensial'nie uravnenia, 20, 12, 1984; On nonlinear equations
of the form $F\left( x,u,Du,\Delta u\right) =0$. Matem. Sb. Ac. Sci. Russ,
1993, v.184, n.11 (Russian Acad. Sci. Sb. Math., 80, (1995) ,2; Solvability
nonlinear equations with operators the form of sum the pseudomonotone and
weakly compact., Soviet Math. Dokl.,1992, v.324, n.5,944-948; Nonlinear
equations in nonreflexive Banach spaces and fully nonlinear equations.
Advances in Mathematical Sciences and Applications, 1999, v.9, n 2, 939-972
(joint with J. Sprekels); On some problem with free boundary. Trans. Russian
Ac. Sci., ser. Math., 2002, 66, 4, 155-176 (joint with Novruzov E.).}
\end{notation}

\begin{corollary}
Assume that the conditions of Theorem 2 are fulfilled and \textit{there is a
continuous function }$\varphi _{1}:R_{+}^{1}\longrightarrow R_{+}^{1}$ such
that $\left\Vert g\left( x\right) \right\Vert _{Y^{\ast }}\leq \varphi
_{1}\left( [x]_{S_{0}}\right) $ for any $x\in X_{0}$ and\textit{\ }$\varphi
\left( \tau \right) \nearrow +\infty $\textit{\ }and $\frac{\varphi \left(
\tau \right) \tau }{\varphi _{1}\left( \tau \right) }\nearrow +\infty $%
\textit{\ as }$\tau \nearrow +\infty $. \textit{Then } \textit{equation
(2.2) is solvable in }$\mathcal{M}_{0}$, \textit{for any }$y\in Y$.
\end{corollary}

\section{Solvability of Problem (1.6) - (1.7)}

A solution of problem (1.6) - (1.7) we will understand in following sense.

\begin{definition}
A function $u\left( t,x\right) $ of the space ${\Large P}_{1,\left(
p-2\right) q,q,2}^{1}\left( Q\right) $ is called a solution of problem (1.6)
- (1.7) if $u\left( t,x\right) $ satisfies the following equality 
\begin{equation}
\left[ \frac{\partial u}{\partial t},v\right] -\left[ \overset{n}{\underset{%
i=1}{\sum }}D_{i}\left( \left\vert D_{i}u\right\vert ^{p-2}D_{i}u\right) ,v%
\right] =\left[ h,v\right] ,\quad \forall v\in L_{p}\left( Q\right) , 
\tag{3.1}
\end{equation}%
where 
\begin{equation*}
{\Large P}_{1,\left( p-2\right) q,q,2}^{1}\left( Q\right) \equiv L_{p}\left(
0,T;\overset{0}{S}{}_{1,\left( p-2\right) q,q}^{1}\left( \Omega \right)
\right) \cap W_{2}^{1}\left( 0,T;L_{2}\left( \Omega \right) \right)
\end{equation*}%
\begin{equation*}
S_{1,\alpha ,\beta }^{1}\left( \Omega \right) \equiv \left\{ u\left(
t,x\right) \left\vert ~\left[ u\right] _{S_{1,\alpha ,\beta }^{1}}^{\alpha
+\beta }=\underset{i=1}{\overset{n}{\sum }}\left\Vert D_{i}u\right\Vert
_{\alpha +\beta }^{\alpha +\beta }+\right. \right.
\end{equation*}%
\begin{equation*}
\left. \underset{i,j=1}{\overset{n}{\sum }}\left\Vert \left\vert
D_{i}u\right\vert ^{\frac{\alpha }{\beta }}D_{j}D_{i}u\right\Vert _{\beta
}^{\beta }<\infty \right\} ,\quad \alpha \geq 0,\ \beta \geq 1.
\end{equation*}%
and $\left[ \cdot ,\cdot \right] $ denotes dual form for the pair $\left(
L_{q}\left( Q\right) ,L_{p}\left( Q\right) \right) $ as in the section 1.
\end{definition}

For the study of problem (1.6) - (1.7) we use Theorem 2 and Corollary 1 of
the previous section. For applying these results to problem (1.6) - (1.7),
we will choose the corresponding spaces and mappings: 
\begin{equation*}
\mathcal{M}_{0}\equiv {\Large P}_{1,\left( p-2\right) q,q,2}^{1}\left(
Q\right) \equiv L_{p}\left( 0,T;\overset{0}{S}{}_{1,\left( p-2\right)
q,q}^{1}\left( \Omega \right) \right) \cap W_{2}^{1}\left( 0,T;L_{2}\left(
\Omega \right) \right) ,
\end{equation*}%
\begin{equation*}
\Phi \left( u\right) \equiv -\ \overset{n}{\underset{i=1}{\sum }}D_{i}\left(
\left\vert D_{i}u\right\vert ^{p-2}D_{i}u\right) ,\quad \gamma _{0}u\equiv
u\left( 0,x\right) ,
\end{equation*}%
\noindent {}%
\begin{equation*}
f\left( \cdot \right) \equiv \left\{ \frac{\partial \cdot }{\partial t}+\Phi
\left( \cdot \right) ;\ \gamma _{0}\cdot \right\} ,\quad g\left( \cdot
\right) \equiv \left\{ \frac{\partial \cdot }{\partial t}-\Delta \cdot
;\quad \gamma _{0}\cdot \right\} ,
\end{equation*}%
\begin{equation*}
X_{0}\equiv W_{p}^{1}\left( 0,T;L_{p}\left( \Omega \right) \right) \cap 
\widetilde{X};\ 
\end{equation*}%
\begin{equation*}
X_{1}\equiv X_{0}\cap \left\{ u\left( t,x\right) \left\vert \frac{\partial u%
}{\partial \widehat{\nu }}\left\vert \ _{\Gamma }\right. =0\right. \right\} ;
\end{equation*}%
\begin{equation*}
Y\equiv L_{q}\left( Q\right) ,\ q=p^{\prime },\widetilde{X}\equiv
L_{p}\left( 0,T;W_{p}^{2}\left( \Omega \right) \right) \cap \left\{ u\left(
t,x\right) \left\vert \frac{\partial u}{\partial \nu }\left\vert \ _{\Gamma
}\right. =0\right. \right\}
\end{equation*}%
here 
\begin{equation*}
\overset{0}{S}{}_{1,\left( p-2\right) q,q}^{1}\left( \Omega \right) \equiv
S_{1,\left( p-2\right) q,q}^{1}\left( \Omega \right) \cap \left\{ u\left(
t,x\right) \left\vert \frac{\partial u}{\partial \widehat{\nu }}\left\vert \
_{\partial \Omega }\right. =0\right. \right\}
\end{equation*}

Thus, as we can see from the above denotations, mapping $f$ is defined by
problem (1.6)-(1.7) and mapping $g$ is defined by the following problem 
\begin{equation}
\frac{\partial u}{\partial t}-\Delta u=v\left( t,x\right) ,\quad \left(
t,x\right) \in Q,  \tag{3.2}
\end{equation}%
\begin{equation}
\gamma _{0}u\equiv u\left( 0,x\right) =u_{0}\left( x\right) ,\ \frac{%
\partial u}{\partial \nu }\left\vert \ _{\Gamma }\right. =0.  \tag{3.3}
\end{equation}

As known (see, [1, 5, 6, 12]), problem (3.2)-(3.3) is solvable in the space 
\begin{equation*}
X_{0}\equiv W_{p}^{1}\left( 0,T;L_{p}\left( \Omega \right) \right) \cap
L_{p}\left( 0,T;W_{p}^{2}\left( \Omega \right) \right) \cap \left\{ u\left(
t,x\right) \left\vert \frac{\partial u}{\partial \nu }\left\vert \ _{\Gamma
}\right. =0\right. \right\}
\end{equation*}%
for any $v\in L_{p}\left( Q\right) $, $u_{0}\in W_{p}^{1}\left( \Omega
\right) $.

Now we will demonstrate that all conditions of Theorem 2 and also of
Corollary 1 are fulfilled.

\begin{proposition}
Mappings $f$ and $g$ , defined above, generate a "coercive" pair on $X_{1}$
in the generalized sense, and moreover the statement of Corollary 1 is valid.
\end{proposition}

\begin{proof}
Let $u\in X_{1}$, i.e. 
\begin{equation*}
u\in X_{0}\cap \left\{ u\left( t,x\right) \left\vert \frac{\partial u}{%
\partial \widehat{\nu }}\left\vert \ _{\Gamma }\right. =0\right. \right\} .
\end{equation*}

Consider the dual form $\left\langle f\left( u\right) ,g\left( u\right)
\right\rangle $ for any $u\in X_{1}$. More exactly, it is enough to consider
the dual form in the form 
\begin{equation}
\underset{0}{\overset{t}{\int }}\underset{\Omega }{\int }f\left( u\right) \
g\left( u\right) \ dxd\tau \equiv \left[ f\left( u\right) ,\ g\left(
u\right) \right] _{t}  \tag{*}
\end{equation}%
Hence, if we consider the above expression then after certain action and in
view of the boundary conditions, we get 
\begin{equation*}
\left[ f\left( u\right) ,\ g\left( u\right) \right] _{t}\equiv \left[ \frac{%
\partial u}{\partial t},\ \frac{\partial u}{\partial t}\right] _{t}\ +\left[
\Phi \left( u\right) ,\ \frac{\partial u}{\partial t}\right] _{t}-
\end{equation*}%
\begin{equation*}
\left[ \frac{\partial u}{\partial t},\ \Delta u\right] _{t}-\left[ \Phi
\left( u\right) ,\ \Delta u\right] _{t}+\underset{\Omega }{\int }u_{0}\
u_{0}\ dx=
\end{equation*}%
\begin{equation*}
=\underset{0}{\overset{t}{\int }}\left\Vert \frac{\partial u}{\partial t}%
\right\Vert _{2}^{2}d\tau +\overset{n}{\underset{i=1}{\sum }}\left[ \frac{1}{%
p}\left\Vert D_{i}u\right\Vert _{p}^{p}\left( t\right) +\frac{1}{2}%
\left\Vert D_{i}u\right\Vert _{2}^{2}\left( t\right) \right] +
\end{equation*}%
\begin{equation*}
\left\Vert u_{0}\right\Vert _{2}^{2}+\left( p-1\right) \overset{n}{\underset{%
i,j=1}{\sum }}\ \underset{0}{\overset{t}{\int }}\left\Vert \left\vert
D_{i}u\right\vert ^{\frac{p-2}{2}}D_{i}D_{j}u\right\Vert _{2}^{2}-
\end{equation*}%
\begin{equation}
\overset{n}{\underset{i=1}{\sum }}\left[ \frac{1}{p}\left\Vert
D_{i}u_{0}\right\Vert _{p}^{p}+\frac{1}{2}\left\Vert D_{i}u_{0}\right\Vert
_{2}^{2}\right]  \tag{3.4}
\end{equation}%
here and in (3.5) we denote $\left\Vert \cdot \right\Vert _{p_{1}}\equiv
\left\Vert \cdot \right\Vert _{L_{p_{1}}\left( \Omega \right) }$, $p_{1}\geq
1$.

From here\ it follows, 
\begin{equation*}
\left[ f\left( u\right) ,\ g\left( u\right) \right] \geq c\left( \left\Vert 
\frac{\partial u}{\partial t}\right\Vert _{L_{2}\left( Q\right) }^{2}+%
\overset{n}{\underset{i=1}{\sum }}\left\Vert \left\vert D_{i}u\right\vert ^{%
\frac{p-2}{2}}D_{i}u\right\Vert _{L_{2}\left( Q\right) }^{2}\right) -
\end{equation*}%
\begin{equation*}
c_{1}\left\Vert u_{0}\right\Vert _{W_{p}^{1}}^{p}-c_{2}\geq \widetilde{c}%
\left( \left\Vert \frac{\partial u}{\partial t}\right\Vert _{L_{2}\left(
Q\right) }^{2}+\left[ u\right] _{L_{p}\left( S_{1,\left( p-2\right)
q,q}\right) }^{p}\right) -
\end{equation*}%
\begin{equation*}
c_{1}\left\Vert u_{0}\right\Vert _{W_{p}^{1}}^{p}-c_{2}\geq \widetilde{c}%
\left[ u\right] _{\mathbf{P}_{1,\left( p-2\right) q,q,2}^{1}\left( Q\right)
}^{2}-c_{1}\left\Vert u_{0}\right\Vert _{W_{p}^{1}}^{p}-\widetilde{c}_{2},
\end{equation*}%
which demonstrates fulfillment of the statement of Corollary 1\footnote{%
From definitions of these spaces is easy to see that $S_{1,p-2,2}^{1}\left(
\Omega \right) \subset S_{1,\left( p-2\right) q,q}\left( \Omega \right) $}.
Consequently, Proposition 1 is true.
\end{proof}

Further for the right part of the dual form, we obtain under the conditions
of Proposition 1 (using same way as in the above proof) 
\begin{equation*}
\left\vert \underset{0}{\overset{t}{\int }}\underset{\Omega }{\int }h\left( 
\frac{\partial u}{\partial t}-\Delta u\right) \ dxd\tau \right\vert \leq
C\left( \varepsilon \right) \underset{0}{\overset{t}{\int }}\left\Vert
h\right\Vert _{2}^{2}\ d\tau +
\end{equation*}%
\begin{equation}
\varepsilon \underset{0}{\overset{t}{\int }}\left\Vert \frac{\partial u}{%
\partial t}\right\Vert _{2}^{2}\ d\tau +C\left( \varepsilon _{1}\right) 
\underset{0}{\overset{t}{\int }}\left\Vert h\right\Vert _{W_{q}^{1}}^{q}\
d\tau +\varepsilon _{1}\overset{n}{\underset{i=1}{\sum }}\underset{0}{%
\overset{t}{\int }}\left\Vert D_{i}u\right\Vert _{p}^{p}\ d\tau .  \tag{3.5}
\end{equation}

It is not difficult to see mapping $g$ defined by problem (3.2)-(3.3)
satisfies of the conditions of \ Theorem 2, i.e. $g\left( X_{1}\right) $
contains an everywhere dense linear manifold of $L_{p}\left( Q\right) $ and $%
g^{-1}$ is weakly compact operator from $L_{p}\left( Q\right) $ to $\mathcal{%
M}_{0}\equiv L_{p}\left( 0,T;\overset{0}{S}{}_{1,\left( p-2\right)
q,q}^{1}\left( \Omega \right) \right) $ $\cap $ $W_{2}^{1}\left(
0,T;L_{2}\left( \Omega \right) \right) $.

Thus we have that all conditions of Theorem 2 and Corollary 1 are fulfilled
for the mappings and spaces corresponding to the studied problem.
Consequently, using Theorem 2 and Corollary 1 we obtain the solvability of
problem (1.6)-(1.7) in the space ${\Large P}_{1,\left( p-2\right)
q,q,2}^{1}\left( Q\right) $ for any $h\in L_{2}\left( 0,T;W_{2}^{1}\left(
\Omega \right) \right) $ and $u_{0}\in W_{p}^{1}\left( \Omega \right) $.

Furthermore, from here it follows that the solution of this problem
possesses the complementary smoothness, i.e. $\overset{n}{\underset{i=1}{%
\sum }}D_{i}\left( \left\vert D_{i}u\right\vert ^{p-2}D_{i}u\right) \in
L_{2}\left( Q\right) $ as far as we have $\frac{\partial u}{\partial t}\in
L_{2}\left( Q\right) $ and $h\in L_{2}\left( 0,T;W_{2}^{1}\left( \Omega
\right) \right) $ by virtue of the conditions of the considered problem.

So the following result is proved.

\begin{theorem}
Under the conditions of this section, problem (1.6) - (1.7) is solvable in $%
{\Large P}\left( Q\right) $ for any $u_{0}\in W_{p}^{1}\left( \Omega \right) 
$ and $h\in L_{2}\left( 0,T;W_{2}^{1}\left( \Omega \right) \right) $ where 
\begin{equation*}
{\Large P}\left( Q\right) \equiv L_{p}\left( 0,T;\overset{0}{\widetilde{S}}%
{}_{1,p-2,2}^{1}\left( \Omega \right) \right) \cap W_{2}^{1}\left(
0,T;L_{2}\left( \Omega \right) \right) \cap {\Large P}_{1,\left( p-2\right)
q,q,2}^{1}\left( Q\right) .
\end{equation*}
\end{theorem}

\section{A priori Estimations for Solutions of Problem (1.4) - (1.5)}

Now, we can investigate the main problem of this article, which is posed for
problem (1.4) - (1.5). We introduce denotations of the mappings $A$ and $f$
\ that are generated by problems (1.4)-(1.5) and (1.6)-(1.7), respectively.

\begin{theorem}
Under the conditions of section 1, any solution $u\left( t,x\right) $ of
problem (1.4) -(1.5) belongs to the bounded subset of the function class $%
\widetilde{P}\left( Q\right) $ defined in the form 
\begin{equation*}
u\in L^{\infty }\left( 0,T;W_{p}^{1}\left( \Omega \right) \right) ;\quad 
\frac{\partial u}{\partial t}\in L^{\infty }\left( 0,T;L_{2}\left( \Omega
\right) \right) ;
\end{equation*}%
\begin{equation}
\overset{n}{\underset{i=1}{\sum }}\overset{t}{\underset{0}{\int }}\left\vert
D_{i}u\right\vert ^{p-2}D_{i}ud\tau \in W_{\infty }^{1}\left(
0,T;L_{q}\left( \Omega \right) \right) \cap L^{\infty }\left(
0,T;W_{2}^{1}\left( \Omega \right) \right)  \tag{4.1}
\end{equation}%
that satisfies the conditions determined by the dates of problem (1.4)-(1.5).
\end{theorem}

\begin{proof}
Consider the dual form $\left\langle A\left( u\right) ,f\left( u\right)
\right\rangle $ for any $u\in {\Large P}\left( Q\right) $ that is defined by
virtue of Theorem 3. We behave as in proof of Proposition 1 and consider
only the integral with respect to $x$. Then we have after certain known acts 
\begin{equation*}
\underset{\Omega }{\int }\frac{\partial u}{\partial t}\ \frac{\partial u}{%
\partial t}\ dx\ +\underset{\Omega }{\int }\ \left( \overset{t}{\underset{0}{%
\int }}\overset{n}{\underset{i=1}{\sum }}D_{i}\left( \left\vert
D_{i}u\right\vert ^{p-2}D_{i}u\right) d\tau \right) \left( \overset{n}{%
\underset{j=1}{\sum }}D_{j}\left( \left\vert D_{j}u\right\vert
^{p-2}D_{j}u\right) \right) \ dx-
\end{equation*}%
\begin{equation*}
\underset{\Omega }{\int }\ \left( \overset{t}{\underset{0}{\int }}\overset{n}%
{\underset{i=1}{\sum }}D_{i}\left( \left\vert D_{i}u\right\vert
^{p-2}D_{i}u\right) d\tau \right) \frac{\partial u}{\partial t}dx-\underset{%
\Omega }{\int }\frac{\partial u}{\partial t}\left( \overset{n}{\underset{i=1}%
{\sum }}D_{i}\left( \left\vert D_{i}u\right\vert ^{p-2}D_{i}u\right) \right)
dx\geq
\end{equation*}%
\begin{equation*}
\frac{1}{2}\left\Vert \frac{\partial u}{\partial t}\right\Vert _{L_{2}\left(
\Omega \right) }^{2}+\frac{1}{2}\frac{\partial }{\partial t}\left\Vert 
\overset{t}{\underset{0}{\int }}\overset{n}{\underset{i=1}{\sum }}%
D_{i}\left( \left\vert D_{i}u\right\vert ^{p-2}D_{i}u\right) d\tau
\right\Vert _{L_{2}\left( \Omega \right) }^{2}+
\end{equation*}%
\begin{equation}
\frac{1}{p}\frac{\partial }{\partial t}\overset{n}{\underset{i=1}{\sum }}%
\left\Vert D_{i}u\right\Vert _{L_{p}\left( \Omega \right) }^{p}-\frac{1}{2}%
\left\Vert \overset{t}{\underset{0}{\int }}\overset{n}{\underset{i=1}{\sum }}%
D_{i}\left( \left\vert D_{i}u\right\vert ^{p-2}D_{i}u\right) d\tau
\right\Vert _{L_{2}\left( \Omega \right) }^{2}\left( t\right)  \tag{4.2}
\end{equation}

Now, consider the right part of the dual form, i.e. $\left\langle
H,f\right\rangle $, for the determination of the bounded subset, to which
the solutions of the problem belongs (and for the receiving of the a priori
estimations). Then we get 
\begin{equation*}
\left\vert \underset{\Omega }{\int }H\ \frac{\partial u}{\partial t}\ dx-%
\underset{\Omega }{\int }H\overset{n}{\underset{j=1}{\sum }}D_{j}\left(
\left\vert D_{j}u\right\vert ^{p-2}D_{j}u\right) \ dx\right\vert \leq
C\left( \varepsilon \right) \left\Vert H\right\Vert _{L_{2}\left( Q\right)
}^{2}+
\end{equation*}%
\begin{equation}
\varepsilon \left\Vert \frac{\partial u}{\partial t}\right\Vert
_{L_{2}\left( \Omega \right) }^{2}\left( t\right) +C\left( \varepsilon
_{1}\right) \left\Vert H\right\Vert _{L_{p}\left( W_{p}^{1}\right)
}^{p}+\varepsilon _{1}\overset{n}{\underset{j=1}{\sum }}\left\Vert
D_{j}u\right\Vert _{L_{p}\left( \Omega \right) }^{p}\left( t\right) . 
\tag{4.3}
\end{equation}

From (4.2) and (4.3) it follows%
\begin{equation*}
0=\underset{\Omega }{\int }\left( A\left( u\right) -H\right) \ f\left(
u\right) \ dx\geq \frac{1}{2}\left\Vert \frac{\partial u}{\partial t}%
\right\Vert _{L_{2}\left( \Omega \right) }^{2}+\frac{1}{p}\frac{\partial }{%
\partial t}\overset{n}{\underset{j=1}{\sum }}\left\Vert D_{j}u\right\Vert
_{L_{p}\left( \Omega \right) }^{p}+
\end{equation*}%
\begin{equation*}
\frac{1}{2}\frac{\partial }{\partial t}\left\Vert \overset{t}{\underset{0}{%
\int }}\overset{n}{\underset{i=1}{\sum }}D_{i}\left( \left\vert
D_{i}u\right\vert ^{p-2}D_{i}u\right) d\tau \right\Vert _{L_{2}\left( \Omega
\right) }^{2}-\frac{1}{2}\left\Vert \overset{t}{\underset{0}{\int }}\overset{%
n}{\underset{i=1}{\sum }}D_{i}\left( \left\vert D_{i}u\right\vert
^{p-2}D_{i}u\right) d\tau \right\Vert _{L_{2}\left( \Omega \right) }^{2}-
\end{equation*}%
\begin{equation*}
\varepsilon _{1}\overset{n}{\underset{i=1}{\sum }}\left\Vert
D_{i}u\right\Vert _{L_{p}\left( \Omega \right) }^{p}-C\left( \varepsilon
\right) \left\Vert H\right\Vert _{L_{2}\left( Q\right) }^{2}-\varepsilon
\left\Vert \frac{\partial u}{\partial t}\right\Vert _{L_{2}\left( \Omega
\right) }^{2}-C\left( \varepsilon _{1}\right) \left\Vert H\right\Vert
_{L_{p}\left( W_{p}^{1}\right) }^{p}
\end{equation*}%
or if we choose small parameters $\varepsilon >0$ and $\varepsilon _{1}>0$
such as needed, then we have 
\begin{equation*}
c\left\Vert \frac{\partial u}{\partial t}\right\Vert _{L_{2}\left( \Omega
\right) }^{2}+\frac{1}{p}\frac{\partial }{\partial t}\overset{n}{\underset{%
i=1}{\sum }}\left\Vert D_{i}u\right\Vert _{L_{p}\left( \Omega \right) }^{p}+
\end{equation*}%
\begin{equation*}
\frac{1}{2}\frac{\partial }{\partial t}\left\Vert \overset{t}{\underset{0}{%
\int }}\overset{n}{\underset{i=1}{\sum }}D_{i}\left( \left\vert
D_{i}u\right\vert ^{p-2}D_{i}u\right) d\tau \right\Vert _{L_{2}\left( \Omega
\right) }^{2}\leq C\left( \left\Vert H\right\Vert _{L_{2}\left( Q\right)
},\left\Vert H\right\Vert _{L_{p}\left( W_{p}^{1}\right) }\right) +
\end{equation*}%
\begin{equation}
\frac{1}{p}\overset{n}{\underset{i=1}{\sum }}\left\Vert D_{i}u\right\Vert
_{L_{p}\left( \Omega \right) }^{p}+\frac{1}{2}\left\Vert \overset{t}{%
\underset{0}{\int }}\overset{n}{\underset{i=1}{\sum }}D_{i}\left( \left\vert
D_{i}u\right\vert ^{p-2}D_{i}u\right) d\tau \right\Vert _{L_{2}\left( \Omega
\right) }^{2}.  \tag{4.4}
\end{equation}%
Inequality (4.4) show that we can use Gronwall lemma. Consequently using
Gronwall lemma we get 
\begin{equation*}
\overset{n}{\underset{i=1}{\sum }}\left\Vert D_{i}u\right\Vert _{L_{p}\left(
\Omega \right) }^{p}\left( t\right) +\left\Vert \overset{t}{\underset{0}{%
\int }}\overset{n}{\underset{i=1}{\sum }}D_{i}\left( \left\vert
D_{i}u\right\vert ^{p-2}D_{i}u\right) d\tau \right\Vert _{L_{2}\left( \Omega
\right) }^{2}\left( t\right) \leq
\end{equation*}%
\begin{equation}
C\left( \left\Vert H\right\Vert _{L_{2}\left( Q\right) },\left\Vert
H\right\Vert _{L_{p}\left( W_{p}^{1}\right) },\left\Vert u_{0}\right\Vert
_{W_{p}^{1}\left( \Omega \right) }\right)  \tag{4.5}
\end{equation}%
holds for a.e. $t\in \left[ 0,T\right] $.

Thus we have for any solution of problem (1.4) -(1.5) the following
estimations 
\begin{equation*}
\left\Vert u\right\Vert _{W_{p}^{1}\left( \Omega \right) }\left( t\right)
\leq C\left( \left\Vert H\right\Vert _{L_{p}\left( W_{p}^{1}\right)
},\left\Vert u_{0}\right\Vert _{W_{p}^{1}\left( \Omega \right) }\right) ,
\end{equation*}%
\begin{equation*}
\left\Vert \overset{t}{\underset{0}{\int }}\overset{n}{\underset{i=1}{\sum }}%
D_{i}\left( \left\vert D_{i}u\right\vert ^{p-2}D_{i}u\right) d\tau
\right\Vert _{L_{2}\left( \Omega \right) }\left( t\right) \leq C\left(
\left\Vert H\right\Vert _{L_{p}\left( W_{p}^{1}\right) },\left\Vert
u_{0}\right\Vert _{W_{p}^{1}\left( \Omega \right) }\right) ,
\end{equation*}%
\begin{equation*}
\left\Vert \frac{\partial u}{\partial t}\right\Vert _{L_{2}\left( \Omega
\right) }\left( t\right) \leq C\left( \left\Vert H\right\Vert _{L_{p}\left(
W_{p}^{1}\right) },\left\Vert u_{0}\right\Vert _{W_{p}^{1}\left( \Omega
\right) }\right)
\end{equation*}%
hold for a.e. $t\in \left[ 0,T\right] $ by virtue of inequalities (4.2) -
(4.5). In other words we have that any solution of problem (1.4) -(1.5)
belongs to the bounded subset of the following class 
\begin{equation*}
u\in L^{\infty }\left( 0,T;W_{p}^{1}\left( \Omega \right) \right) ;\quad 
\frac{\partial u}{\partial t}\in L^{\infty }\left( 0,T;L_{2}\left( \Omega
\right) \right) ;
\end{equation*}%
\begin{equation*}
\frac{\partial }{\partial t}\left( \overset{n}{\underset{i=1}{\sum }}\overset%
{t}{\underset{0}{\int }}\left\vert D_{i}u\right\vert ^{p-2}D_{i}ud\tau
\right) \in L^{\infty }\left( 0,T;L_{q}\left( \Omega \right) \right)
\end{equation*}%
\begin{equation}
\overset{t}{\underset{0}{\int }}\overset{n}{\underset{i=1}{\sum }}%
D_{i}\left( \left\vert D_{i}u\right\vert ^{p-2}D_{i}u\right) d\tau \in
L^{\infty }\left( 0,T;L_{2}\left( \Omega \right) \right) ,  \tag{4.6}
\end{equation}%
for each given $u_{0},u_{1}\in W_{p}^{1}\left( \Omega \right) $, $h\in
L_{p}\left( 0,T;W_{p}^{1}\left( \Omega \right) \right) $.

From here it follows that all solutions of this problem belong to a bounded
subset of space $P\left( Q\right) $, which is defined by (4.1).

Indeed, firstly it is easy to see that the following inequality holds 
\begin{equation*}
\left\Vert \overset{n}{\underset{i=1}{\sum }}\overset{t}{\underset{0}{\int }}%
\left\vert D_{i}u\right\vert ^{p-2}D_{i}ud\tau \right\Vert _{L_{q}\left(
\Omega \right) }^{q}\leq C\overset{n}{\underset{i=1}{\sum }}\left\Vert
\left\vert D_{i}u\right\vert ^{p-2}D_{i}u\right\Vert _{L_{q}\left( \Omega
\right) }^{q}\leq
\end{equation*}%
\begin{equation*}
C\left( T,mes\ \Omega \right) \left\Vert u\right\Vert _{W_{p}^{1}\left(
\Omega \right) }^{q}\left( t\right) \Longrightarrow \overset{t}{\underset{0}{%
\int }}\overset{n}{\underset{i=1}{\sum }}\left\vert D_{i}u\right\vert
^{p-2}D_{i}ud\tau \in L^{\infty }\left( 0,T;L_{q}\left( \Omega \right)
\right) ,
\end{equation*}%
and secondary, the following equalities are correct%
\begin{equation*}
\underset{\Omega }{\int }\left( \overset{t}{\underset{0}{\int }}\overset{n}{%
\underset{i=1}{\sum }}D_{i}\left( \left\vert D_{i}u\right\vert
^{p-2}D_{i}u\right) d\tau \right) ^{2}dx\equiv \left\Vert \overset{t}{%
\underset{0}{\int }}\overset{n}{\underset{i=1}{\sum }}D_{i}\left( \left\vert
D_{i}u\right\vert ^{p-2}D_{i}u\right) d\tau \right\Vert _{2}^{2}\equiv
\end{equation*}%
\begin{equation*}
\left\langle \overset{t}{\underset{0}{\int }}\overset{n}{\underset{i=1}{\sum 
}}D_{i}\left( \left\vert D_{i}u\right\vert ^{p-2}D_{i}u\right) d\tau ,%
\overset{t}{\underset{0}{\int }}\overset{n}{\underset{j=1}{\sum }}%
D_{j}\left( \left\vert D_{j}u\right\vert ^{p-2}D_{j}u\right) d\tau
\right\rangle =
\end{equation*}%
\begin{equation*}
\overset{n}{\underset{i,j=1}{\sum }}\left\langle \overset{t}{\underset{0}{%
\int }}D_{j}\left( \left\vert D_{i}u\right\vert ^{p-2}D_{i}u\right) d\tau ,%
\overset{t}{\underset{0}{\int }}D_{i}\left( \left\vert D_{j}u\right\vert
^{p-2}D_{j}u\right) d\tau \right\rangle =
\end{equation*}%
\begin{equation*}
\overset{n}{\underset{i,j=1}{\sum }}\left\langle D_{j}\overset{t}{\underset{0%
}{\int }}\left\vert D_{i}u\right\vert ^{p-2}D_{i}ud\tau ,D_{i}\overset{t}{%
\underset{0}{\int }}\left\vert D_{j}u\right\vert ^{p-2}D_{j}ud\tau
\right\rangle ,
\end{equation*}%
and also 
\begin{equation*}
\overset{n}{\underset{j=1}{\sum }}\left\Vert D_{j}\overset{t}{\underset{0}{%
\int }}~\overset{n}{\underset{i=1}{\sum }}\left( \left\vert
D_{i}u\right\vert ^{p-2}D_{i}u\right) d\tau \right\Vert _{2}^{2}=
\end{equation*}%
\begin{equation*}
\overset{n}{\underset{j=1}{\sum }}\left\langle D_{j}\overset{t}{\underset{0}{%
\int }}~\overset{n}{\underset{i=1}{\sum }}\left\vert D_{i}u\right\vert
^{p-2}D_{i}ud\tau ,D_{j}\overset{t}{\underset{0}{\int }}~\overset{n}{%
\underset{i=1}{\sum }}\left\vert D_{i}u\right\vert ^{p-2}D_{i}ud\tau
\right\rangle .
\end{equation*}%
These demonstrate that the function 
\begin{equation*}
v\left( t,x\right) \equiv \overset{t}{\underset{0}{\int }}\ \overset{n}{%
\underset{i=1}{\sum }}\left\vert D_{i}u\right\vert ^{p-2}D_{i}u\ d\tau
\end{equation*}%
belongs to a bounded subset of the space 
\begin{equation*}
L^{\infty }\left( 0,T;L_{q}\left( \Omega \right) \right) \cap \left\{
v\left( t,x\right) \left\vert ~Dv\in \right. L^{\infty }\left(
0,T;L_{2}\left( \Omega \right) \right) \right\} .
\end{equation*}%
Therefore, in order to prove the correctness of (4.1), it remains to use the
following inequality, i.e. the Nirenberg-Gagliardo-Sobolev inequality 
\begin{equation}
\left\Vert D^{\beta }v\right\Vert _{p_{2}}\leq C\left( \underset{\left\vert
\alpha \right\vert =m}{\sum }\left\Vert D^{\alpha }v\right\Vert
_{p_{0}}^{\theta }\right) \left\Vert v\right\Vert _{p_{1}}^{1-\theta },\quad
0\leq \left\vert \beta \right\vert =l\leq m-1,  \tag{4.7}
\end{equation}%
which holds for each $v\in W_{p_{0}}^{m}\left( \Omega \right) $, $\Omega
\subset R^{n}$, $n\geq 1$, $C\equiv C\left( p_{0},p_{1},p_{2},l,s\right) $
and $\theta $ such that $\frac{1}{p_{2}}-\frac{l}{n}=\left( 1-\theta \right) 
\frac{1}{p_{1}}+\theta \left( \frac{1}{p_{0}}-\frac{m}{n}\right) $. Really,
in inequality (4.7) for us it is enough to choose $p_{2}=2$, $l=0$, $p_{1}=q$%
, $p_{0}=2$ then we get 
\begin{equation*}
\frac{1}{2}=\left( 1-\theta \right) \frac{p-1}{p}+\theta \left( \frac{1}{2}-%
\frac{1}{n}\right) \Longrightarrow \theta \left( \frac{1}{2}-\frac{1}{n}-%
\frac{p-1}{p}\right) =\frac{1}{2}-\frac{p-1}{p}\Longrightarrow
\end{equation*}%
$\theta =\frac{n\left( p-2\right) }{n\left( p-2\right) +2p}$ for $p>2$, and
so (4.1) is correct.
\end{proof}

\begin{corollary}
Under the conditions of Theorem 4, each solution of problem (1.1)-(1.3)
belongs to the bounded subset of the class $\mathbf{V}\left( Q\right) $
defined in (DS).
\end{corollary}

\begin{proof}
From (4.1) it follows 
\begin{equation*}
\overset{n}{\underset{i=1}{\sum }}\overset{t}{\underset{0}{\int }}\left\vert
D_{i}u\right\vert ^{p-2}D_{i}ud\tau \in L^{\infty }\left( 0,T;\overset{0}{W}%
\ _{2}^{1}\left( \Omega \right) \right) \cap W_{\infty }^{1}\left(
0,T;L_{q}\left( \Omega \right) \right) ,
\end{equation*}%
moreover 
\begin{equation*}
\overset{t}{\underset{0}{\int }}\overset{n}{\underset{i,j=1}{\sum }}%
D_{j}\left( \left\vert D_{i}u\right\vert ^{p-2}D_{i}u\right) d\tau \in
L^{\infty }\left( 0,T;L_{2}\left( \Omega \right) \right) ,
\end{equation*}%
and is bounded in this space. Then taking into account the property of the
Lebesgue integrals we obtain 
\begin{equation*}
\overset{t}{\underset{0}{\int }}\ \left\{ \underset{\Omega }{\int }\left[ 
\overset{n}{\underset{i,j=1}{\sum }}D_{j}\left( \left\vert D_{i}u\right\vert
^{p-2}D_{i}u\right) \right] ^{2}dx\right\} ^{\frac{1}{2}}d\tau \leq C,\quad
C\neq C\left( t\right)
\end{equation*}

from which we get 
\begin{equation*}
\overset{n}{\underset{i,j=1}{\sum }}D_{j}\left( \left\vert D_{i}u\right\vert
^{p-2}D_{i}u\right) \in L_{1}\left( 0T;L_{2}\left( \Omega \right) \right) ,
\end{equation*}%
and so 
\begin{equation}
\overset{n}{\underset{i=1}{\sum }}D_{i}\left( \left\vert D_{i}u\right\vert
^{p-2}D_{i}u\right) \in L_{1}\left( 0T;L_{2}\left( \Omega \right) \right) 
\tag{4.8}
\end{equation}%
in which is bounded.

If we consider equation (1.1), and take into account that it is solvable in
the generalized sense and $\frac{\partial u}{\partial t}\in W_{\infty
}^{1}\left( 0,T;L_{2}\left( \Omega \right) \right) $ (by (4.1)) then from
Definition 1 it follows that 
\begin{equation*}
\left[ \frac{\partial ^{2}u}{\partial t^{2}},v\right] -\left[ \overset{n}{%
\underset{i=1}{\sum }}D_{i}\left( \left\vert D_{i}u\right\vert
^{p-2}D_{i}u\right) ,v\right] =\left[ h,v\right]
\end{equation*}%
holds for any $v\in W_{\widetilde{p}}^{1}\left( 0,T;L_{2}\left( \Omega
\right) \right) $, $\widetilde{p}>1$.

Hence 
\begin{equation}
\left[ \frac{\partial ^{2}u}{\partial t^{2}},v\right] =\left[ \overset{n}{%
\underset{i=1}{\sum }}D_{i}\left( \left\vert D_{i}u\right\vert
^{p-2}D_{i}u\right) +h,v\right]  \tag{4.9}
\end{equation}%
holds for any $v\in L^{\infty }\left( Q\right) $.

Thus we obtain $\frac{\partial ^{2}u}{\partial t^{2}}\in L_{1}\left(
0T;L_{2}\left( \Omega \right) \right) $ by virtue of (4.1), (4.8) and as 
\begin{equation*}
\overset{n}{\underset{i=1}{\sum }}D_{i}\left( \left\vert D_{i}u\right\vert
^{p-2}D_{i}u\right) +h\in L_{1}\left( 0T;L_{2}\left( \Omega \right) \right) .
\end{equation*}
\end{proof}

\end{document}